\documentclass[fleqn,preprint,3p,a4paper]{elsarticle}

\usepackage{amssymb}
\usepackage{amsmath}
\usepackage{amsthm}
\usepackage{dcolumn}
\usepackage{endnotes}
\usepackage{tabularx}
\usepackage[matrix,arrow]{xy}
\usepackage{wasysym}
\usepackage{float}
\usepackage{graphicx}

\newtheorem{theorem}{Theorem}[section]
\newtheorem{proposition}[theorem]{Proposition}
\newtheorem{lemma}[theorem]{Lemma}
\newtheorem{corollary}[theorem]{Corollary}

\theoremstyle{definition}
\newtheorem{definition}[theorem]{Definition}
\newtheorem{example}[theorem]{Example}
\newtheorem{remark}[theorem]{Remark}
\newtheorem{question}[theorem]{Question}

\newcommand{\ir}{{\mathsf{Irr}}}

\newcommand{\mn}{\mathbb N}

\newcommand{\cl}{{\rm cl}}

\newcommand{\ua}{\mathord{\uparrow}}
\newcommand{\da}{\mathord{\downarrow}}

\begin{document}

\begin{frontmatter}



\journal{ }

\title{On function spaces related to H-sober spaces\tnoteref{t1}}
\tnotetext[t1]{This research is supported by the National Natural Science Foundation of China (Nos. 12071199, 11661057) and the Natural Science Foundation of Jiangxi Province (No. 20192ACBL20045).}


\author[M. Bao]{Meng Bao}
\ead{mengbao95213@163.com}
\address[M. Bao]{College of Mathematics, Sichuan University, Chengdu 610064, China}

\author[X. Zhang]{Xiaoyuan Zhang}
\ead{zxy0407010229@163.com}
\address[X. Zhang]{College of Mathematics,
Sichuan University, Chengdu 610064, China\\
School of Big Data Science, Hebei Finance University, Baoding 071051, China}

\author[X. Xu]{Xiaoquan Xu\corref{mycorrespondingauthor}}
\cortext[mycorrespondingauthor]{Corresponding author.}
\ead{xiqxu2002@163.com}
\address[X. Xu]{Fujian Key Laboratory of Granular Computing and Applications, Minnan Normal University, Zhangzhou 363000, China}

\begin{abstract}
In this paper, we mainly study the function spaces related to H-sober spaces. For an irreducible subset system H and $T_{0}$ spaces $X$ and $Y$, it is proved that $Y$ is H-sober iff the function space $\mathbb{C}(X, Y)$ of all continuous functions $f : X\longrightarrow Y$ equipped with the topology of pointwise convergence is H-sober iff the function space $\mathbb{C}(X, Y)$ equipped with the Isbell topology is H-sober. One immediate corollary is that for a $T_{0}$ space $X$, $Y$ is a sober space (resp., $d$-space, well-filtered space) iff the function space $\mathbb{C}(X, Y)$ equipped with the topology of pointwise convergence is a sober space (resp., $d$-space, well-filtered space) iff the function space $\mathbb{C}(X, Y)$ equipped with the the Isbell topology is a sober space (resp., $d$-space, well-filtered space). It is shown that $T_{0}$ spaces $X$ and $Y$, if the function space $\mathbb{C}(X, Y)$ equipped with the compact-open topology is H-sober, then $Y$ is H-sober. The function space $\mathbb{C}(X, Y)$ equipped with the Scott topology is also discussed.
\end{abstract}


\end{frontmatter}

\section{Introduction}
Function spaces (equipped with certain topologies) are important structures in topology and domain theory, which was initially introduced by Dana Scott \cite{Sc}. As a special kind of mathematical structures, domains serve as mathematical
universes within which people can interpret higher-order functional programming languages and cartesian closed categories of domains (more generally, certain topological spaces) are appropriate for models of various typed and untyped lambda-calculi
and functional programming languages (see \cite{GHKLMS}). Since whether certain properties of topological spaces are preserved when passing to function spaces is connected with the cartesian closed category of topological spaces, this question has attracted considerable attention in domain theory and non-Hausdorff topology, especially for domains (which are a special kind of topological spaces when endowed with the Scott topology), sober spaces, $d$-spaces and well-filtered spaces.

 There exists a quite satisfactory theory which deals with the cartesian closedness
of domains. Jung \cite{Jung89, Jung90}, Plotkin \cite{Plotkin76} and Smyth \cite{Smyth83} have made essential contributions to this theory. For topological spaces, it is well-known that if $X$ is a $T_0$ and $Y$ a sober space, then the function space $\mathbb{C}(X, Y)$ of all continuous functions $f:X\rightarrow Y$ equipped with the topology of pointwise convergence is sober (see, for example, \cite[Exercise O-5.16]{GHKLMS}). Furthermore, in \cite{ES}, it was shown that for any $T_{0}$ space $X$, a $T_0$ space $Y$ is a $d$-space (resp., sober space) iff the the function space $\mathbb{C}(X, Y)$ equipped with the topology of pointwise convergence is a $d$-space (resp., sober space). It is known that for a $T_0$ space $X$ and a $d$-space $Y$, the function space $\mathbb{C}(X, Y)$ equipped with the Isbell topology is a $d$-space (cf. \cite[Lemma II-4.3]{GHKLMS}). Conversely, in \cite{LLH}, the authors shown that if the function space $\mathbb{C}(X, Y)$ equipped with the Isbell topology is a $d$-space, then $Y$ is a $d$-space. For the well-filteredness, it was proved in \cite{LLH} that for any core compact space $X$ and well-filtered space $Y$, the function space $\mathbb{C}(X, Y)$ equipped with the Isbell topology is well-filtered.

 In order to provide a uniform approach to $d$-spaces, sober spaces and well-filtered spaces, Xu \cite{Xu2} introduced the concepts of irreducible subset system H and H-sober spaces and developed a general framework for dealing with all these spaces. The irreducible subset systems $\mathcal D$, $\mathcal R$, $\mathsf{WD}$ and $\mathsf{RD}$ are four important ones, where for each $T_0$ space $X$, $\mathcal D(X)$ is the set of all directed subsets of $X$, $\mathcal R(X)$ is the set of all irreducible subsets of $X$, $\mathsf{WD}(X)$ is the set of all well-filtered determined subsets of $X$ and $\mathsf{RD}(X)$ is the set of all Rudin subsets of $X$ (see \cite{Xu2, XSXZ}). So the $d$-spaces, sober spaces and well-filtered spaces are three special types of H-sober spaces. It was proved in \cite{Xu2} that for a $T_0$ space $X$ and an H-sober space $Y$, the function space $\mathbb{C}(X, Y)$ equipped with the topology of pointwise convergence is H-sober. One immediate corollary is that for a $T_0$ space $X$ and a sober space (resp., $d$-space, well-filtered space) $Y$, the function space $\mathbb{C}(X, Y)$ equipped with the topology of pointwise convergence is a sober space (resp., $d$-space, well-filtered space).

In this paper, we mainly study the function spaces related to H-sober spaces. For an irreducible subset system H and $T_{0}$ spaces $X$ and $Y$, it is proved that the following three conditions are equivalent:
\begin{enumerate}[\rm (1)]
\item $Y$ is H-sober.
 \item The function space $\mathbb{C}(X, Y)$ equipped with the pointwise convergence topology is H-sober.
  \item The function space $\mathbb{C}(X, Y)$ equipped with the Isbell topology is H-sober.
  \end{enumerate}

  \noindent Applying this result directly to the irreducible subset systems $\mathcal D$, $\mathcal R$ and $\mathsf{WD}$ (or $\mathsf{RD}$), we get the conclusion that for any $T_{0}$ spaces $X$ and $Y$, the following three conditions are equivalent: (1) $Y$ is a sober space (resp., $d$-space, well-filtered space); (2) the function space $\mathbb{C}(X, Y)$ equipped with the topology of pointwise convergence is a sober space (resp., $d$-space, well-filtered space); (3) the function space $\mathbb{C}(X, Y)$ equipped with the Isbell topology is a sober space (resp., $d$-space, well-filtered space). This conclusion improves the result in \cite[Corollary 3.12]{LLH} that for any core compact space $X$, if $Y$ is well-filtered, then the function space $\mathbb{C}(X, Y)$ equipped with the Isbell topology is well-filtered. It is shown that for $T_{0}$ spaces $X$ and $Y$, if the function space $\mathbb{C}(X, Y)$ equipped with the compact-open topology is H-sober, then $Y$ is H-sober. Therefore, for any $T_{0}$ spaces $X$ and $Y$, $Y$ is a $d$-space iff the function space $\mathbb{C}(X, Y)$ equipped with the compact-open topology is a $d$-space. The function space $\mathbb{C}(X, Y)$ equipped with the Scott topology is also discussed.

\section{Preliminaries}
In this section, we briefly recall some fundamental concepts and notations that will be used in this paper, more details can be founded in \cite{RE, GHKLMS, GL}.

For a poset $P$ and $A \subseteq P$, let $\ua A=\{x\in P: a\leq x \mbox{ for some }a\in A \}$ (dually $\da A=\{x \in P: x \leq a \mbox{ for some }a\in A\})$. For $A=\{x\}$, $\ua A$ and $\da A$ are shortly denoted by $\ua x$ and $\da x$ respectively. A subset $A$ is called a \emph{lower set} (resp., an \emph{upper set}) if $A=\da A$ (resp., $A=\ua A$). Let $P^{(<\omega)}=\{F\subseteq P : F \mbox{~is a finite set}\}$ and $\mathbf{Fin} P=\{\uparrow F : \emptyset\neq F\in P^{(<\omega)}\}$. A subset $D$ of $P$ is \emph{directed} provided that it is nonempty and every finite subset of $D$ has an upper bound in $D$. The set of all directed sets of $P$ is denoted by $\mathcal{D}(P)$. $P$ is said to be a \emph{directed complete poset}, a \emph{dcpo} for short, if every directed subset of $P$ has the least upper bound in $P$. As in \cite{GHKLMS}, the \emph{upper topology} on a poset $P$, generated
by the complements of the principal ideals of $P$, is denoted by $\upsilon (P)$. The upper sets of $P$ form the (\emph{upper}) \emph{Alexandroff topology} $\gamma (P)$. The space $\Gamma~\!\!P=(P, \gamma(P))$ is called the \emph{Alexandroff space} of $P$. A subset $U$ of a poset $P$ is called \emph{Scott open} if $U=\mathord{\uparrow} U$ and $D\cap U\neq\emptyset$ for all directed sets $D\subseteq P$ with $\vee D\in U$ whenever $\vee D$ exists. The topology formed by all Scott open sets of $P$ is called the \emph{Scott topology}, written as $\sigma(P)$. $\Sigma P=(P, ~\!\sigma(P))$ is called the \emph{Scott space} of $P$. Clearly, $\Sigma P$ is a $T_0$ space.

Given a $T_0$ space $X$, we can define a partially order $\leq_{X}$, called the \emph{specialization order}, which is defined by $x \leq_X y$ iff $x \in \overline{\{y\}}$. Let $\Omega X$ denote the poset $(X, \leq_X)$. Clearly, each open set is an upper set and each closed set is a lower set with respect to the partially order $\leq_{X}$. Unless otherwise stated, throughout the paper, whenever an order-theoretic concept is mentioned in a $T_0$ space, it is to be interpreted with respect to the specialization order. Let $\mathcal{O}(X)$ (resp., $\mathcal{C}(X)$) be the set of all open subsets (resp., closed subsets) of $X$ and denote $\mathcal{S}(X)=\{\{x\}: x\in X\}$, $\mathcal{D}(X)=\{D\subseteq X: D$ is a directed set of $X\}$. A {\it retract} of $X$ is a topological space $Y$ such that there are two continuous mappings $f:X\rightarrow Y$ and $g:Y\rightarrow X$ such that $f\circ g=id_{Y}$.

The category of all sets and mappings is denoted by $\mathbf{Set}$ and the category of all $T_0$ spaces with continuous mappings is denoted by $\mathbf{Top}_0$. A $T_0$ space $X$ is called a $d$-\emph{space} (or \emph{monotone convergence space}) if $X$ (with the specialization order) is a dcpo and $\mathcal O(X) \subseteq \sigma(X)$ (cf.\cite{GHKLMS}). Clearly, for a dcpo $P$, $\Sigma~\!\! P$ is a $d$-space.

One can directly get the following result (cf. \cite{GHKLMS, XSXZ}).

\begin{proposition}\label{2mt1} For a $T_0$ space $X$, the following conditions are equivalent:
\begin{enumerate}[\rm (1)]
	        \item $X$ is a $d$-space.
            \item $\mathcal D_c(X)=\mathcal S_c(X)$, that is, for each $D\in \mathcal D(X)$, there exists a \emph{(}unique\emph{)} point $x\in X$ such that $\overline{D}=\overline{\{x\}}$.
\end{enumerate}
\end{proposition}

A nonempty subset $A$ of a $T_0$ space $X$ is called \emph{irreducible} if for any $\{F_1, F_2\}\subseteq \mathcal{C}(X)$, $A\subseteq F_1\cup F_2$ implies $A\subseteq F_1$ or $A\subseteq F_2$. We denote by $\ir(X)$ (resp., $\ir_c(X)$) the set of all irreducible (resp., irreducible closed) subsets of $X$. Clearly, every subset of $X$ that is directed under $\leq_{X}$ is irreducible. A topological space $X$ is called \emph{sober}, if for any $ F\in\ir_c(X)$, there is a unique point $x\in X$ such that $F=\overline{\{x\}}$.

For a topological space $X$, $\mathcal{G}\subseteq 2^{X}$ and $A\subseteq X$, let $\Diamond_{\mathcal{G}}A=\{G\in \mathcal{G}: G\cap A\neq\emptyset\}$ and $\Box_{\mathcal{G}}A=\{G\in \mathcal{G}: G\subseteq A\}$. The sets $\Diamond_{\mathcal{G}}A$ and $\Box_{\mathcal{G}}A$ will be simply written as $\Diamond A$ and $\Box A$ respectively if there is no confusion. The \emph{lower Vietoris topology} on $\mathcal{G}$ is the topology that has $\{\Diamond U: U\in \mathcal{O}(X)\}$ as a subbase, and the resulting space is denoted by $P_{H}(\mathcal{G})$. If $\mathcal{G}\subseteq \ir (X)$, then $\{\Diamond_{\mathcal{G}}U: U\in \mathcal{O}(X)\}$ is a topology on $\mathcal{G}$.

\begin{remark} \label{2zj1} Let $X$ be a $T_0$ space.
\begin{enumerate}[\rm (1)]
	\item If $\mathcal{S}_c(X)\subseteq \mathcal{G}$, then the specialization order on $P_H(\mathcal{G})$ is the order of set inclusion, and the \emph{canonical mapping} $\eta_{X}: X\longrightarrow P_H(\mathcal{G})$, given by $\eta_X(x)=\overline {\{x\}}$, is an order and topological embedding.
    \item The space $X^s=P_H(\ir_c(X))$ with the canonical mapping $\eta_{X}: X\longrightarrow X^s$ is the \emph{sobrification} of $X$ (cf. \cite{GHKLMS, GL}).
\end{enumerate}
\end{remark}

For a $T_0$ space $X$, let $2^{X}$ be the set of all subsets of $X$. A subset $A$ of $X$ is called \emph{saturated} if $A$ equals the intersection of all open sets containing it (equivalently, $A$ is an upper set in the specialization order). Let $\mathcal S^u(X)$ denote the set of all principal filters, namely, $\mathcal S^u(X)=\{\ua x : x\in X\}$. We denote by $\mathord{\mathsf{K}}(X)$ the poset of nonempty compact saturated subsets of $X$ with the \emph{Smyth preorder}, i.e., for $K_{1}, K_{2}\in \mathord{\mathsf{K}}(X)$, $K_{1}\sqsubseteq K_{2}$ iff $K_{2}\subseteq K_{1}$.  A $T_0$ space $X$ is called \emph{well-filtered}, if for any open set $U$ and any filtered family $\mathcal{K}\subseteq \mathord{\mathsf{K}}(X)$, $\bigcap\mathcal{K}{\subseteq} U$ implies $K{\subseteq} U$ for some $K{\in}\mathcal{K}$. We consider the upper Vietoris topology on $\mathord{\mathsf{K}}(X)$, generated by the sets $\Box U=\{K\in \mathord{\mathsf{K}}(X): K\subseteq U\}$, where $U$ ranges over the open subsets of $X$. The resulting space is called the \emph{Smyth power space} or \emph{upper space} of $X$ and denoted by $P_{S}(X)$.

\begin{remark} \label{2zj2} Let $X$ be a $T_0$ space. Then
\begin{enumerate}[\rm (1)]
	\item the specialization order on $P_{S}(X)$ is the Smyth order, that is, $\leq_{P_{S}(X)}=\sqsubseteq$.
    \item the \emph{canonical mapping} $\xi_{X}: X\longrightarrow P_{S}(X), x \mapsto \ua x$, is an order and topological embedding (cf. \cite{Schalk}).
\end{enumerate}
\end{remark}

\section{Topological Rudin Lemma and H-sober spaces}

Rudin's Lemma plays a crucial role in domain theory and is a useful tool in studying the various aspects of well-filtered spaces and sober spaces. In \cite{HK}, Heckman and Keimel presented the following topological variant of Rudin's Lemma.

\begin{lemma}\label{2yl4} \emph{(Topological Rudin Lemma)}  Let $X$ be a topological space and $\mathcal{A}$ an
irreducible subset of the Smyth power space $P_S(X)$. Then every closed set $C \subseteq X$  that
meets all members of $\mathcal{A}$ contains a minimal irreducible closed subset $A$ that still meets all
members of $\mathcal{A}$.
\end{lemma}

For a $T_0$ space $X$ and $\mathcal{K}\subseteq \mathord{\mathsf{K}}(X)$, let $M(\mathcal{K})=\{A\in \mathcal{C}(X): K\cap A\neq\emptyset \mbox{~for all~} K\in \mathcal{K}\}$ (that is, $\mathcal K\subseteq \Diamond A$) and $m(\mathcal{K})=\{A\in \mathcal{C}(X): A \mbox{~is a minimal menber of~} M(\mathcal{K})\}$.

\begin{definition}\label{2dy10} (\cite{SXXZ, XSXZ})
Let $X$ be a $T_0$ space and $A$ a nonempty subset of $X$.
\begin{enumerate}[\rm (1)]
\item $A$ is said to have the \emph{Rudin property} (which is called \emph{compactly filtered property} in \cite{SXXZ}), if there exists a filtered family $\mathcal K\subseteq \mathord{\mathsf{K}}(X)$ such that $\overline{A}\in m(\mathcal{K})$ (that is, $\overline{A}$ is a minimal closed set that intersects all members of $\mathcal{K}$). Let $\mathsf{RD}(X)=\{A\subseteq X: A \mbox{{~has Rudin property}}\}$ and $\mathsf{RD}_{c}(X)=\{\overline{A}: A \in \mathsf{RD}(X)\}$.  The sets in $\mathsf{RD}(X)$ will also be called \emph{Rudin sets}. Obviously, a subset $A$ of a space $X$ has Rudin property iff $\overline{A}$ has Rudin property.
\item $A$ is called a \emph{well-filtered determined set} if for any continuous mapping $f:X\longrightarrow Y$ into a well-filtered space $Y$, there exists a unique $y_{A}\in Y$ such that $\overline{f(A)}=\overline{\{y_{A}\}}$. Denote by $\mathsf{WD}(X)$ the set of all well-filtered determined subsets of $X$ and $\mathsf{WD}_{c}(X)=\{\overline{A}: A\in \mathsf{WD}(X)\}$. Obviously, a subset $A$ of a space $X$ is well-filtered determined iff $\overline{A}$ is well-filtered determined.
\end{enumerate}
\end{definition}

\begin{proposition}\label{2mt6}\emph{(\cite[Proposition 6.2]{XSXZ})} Let $X$ be a $T_{0}$ space. Then $\mathcal{S}(X)\subseteq \mathcal{D}(X)\subseteq \mathsf{RD}(X)\subseteq \mathsf{WD}(X)\subseteq \ir(X)$.
\end{proposition}

\begin{proposition}\label{2mt7}\emph{(\cite[Lemma 2.5]{SXXZ}, \cite[Lemma 6.23]{XSXZ})} Let $X, Y$ be two $T_{0}$ spaces and $f:X\longrightarrow Y$ a continuous mapping. Then the following statements hold:
\begin{enumerate}[\rm (1)]
\item If $A\in \mathsf{WD}(X)$, then $f(A)\in \mathsf{WD}(Y)$.
\item If $A\in \mathsf{RD}(X)$, then $f(A)\in \mathsf{RD}(Y)$.
\end{enumerate}
\end{proposition}

Based on the Rudin sets and well-filtered determined sets, the following characterization of well-filtered spaces was given in \cite{XSXZ}.

\begin{proposition}\label{2mt8}\emph{(\cite[Corollary 7.11]{XSXZ})} Let $X$ be a $T_{0}$ space. Then the following conditions are equivalent:
\begin{enumerate}[\rm (1)]
\item $X$ is well-filtered.
\item $\mathsf{WD}_{c}(X)=\mathcal{S}_{c}(X)$, that is, for any $A\in \mathsf{WD}(X)$, there exists a unique $a\in X$ such that $\overline{A}=\overline{\{a\}}$.
\item $\mathsf{RD}_{c}(X)=\mathcal{S}_{c}(X)$, that is, for any $A\in \mathsf{RD}(X)$, there exists a unique $a\in X$ such that $\overline{A}=\overline{\{a\}}$.
\end{enumerate}
\end{proposition}

In order to provide a uniform approach to $d$-spaces, sober spaces and well-filtered spaces and develop a general framework for dealing with all these spaces, Xu \cite{Xu2} introduced the following concepts.

\begin{definition}\label{2dy3} (\cite{Xu2}) (1) A covariant functor ${\rm H}: \mathbf{Top}_0 \longrightarrow \mathbf{Set}$ is called a \emph{subset system} on $\mathbf{Top}_0$ provided that the following two conditions are satisfied:
\begin{enumerate}[\rm (i)]
\item $\mathcal S(X)\subseteq {\rm H}(X)\subseteq 2^X$ (the set of all subsets of $X$) for each $X\in$ \emph{ob}($\mathbf{Top}_0$).
\item For any continuous mapping $f: X \longrightarrow Y$ in $\mathbf{Top}_0$, ${\rm H}(f)(A)=f(A)\in {\rm H}(Y)$ for all $A\in {\rm H}(X)$.
\end{enumerate}
(2) A subset system ${\rm H}: \mathbf{Top}_0 \longrightarrow \mathbf{Set}$ is called an \emph{irreducible subset system}, or an \emph{R-subset system} for short, if ${\rm H}(X)\subseteq \ir (X)$ for all $X\in$ \emph{ob}($\mathbf{Top}_0$).
\end{definition}

In what follows, the capital letter H always stands for an R-subset system ${\rm H} : \mathbf{Top}_0 \longrightarrow \mathbf{Set}$. For a $T_0$ space $X$, let ${\rm H}_c(X)=\{\overline{A} : A\in {\rm H}(X)\}$. Define a partially order $\leq$ on the set of all R-subset systems by ${\rm H}_1\leq {\rm H}_2$ iff $ {\rm H}_1(X)\subseteq {\rm H}_2(X)$ for all $T_0$ spaces $X$.

Here are some important examples of R-subset systems used later:
\begin{enumerate}[\rm (1)]
    \item $\mathcal S$ (for $X\in$ \emph{ob}($\mathbf{Top}_0$), $\mathcal S(X)$ is the set of all single point subsets of $X$).
    \item $\mathcal D$ (for $X\in$ \emph{ob}($\mathbf{Top}_0$), $\mathcal D(X)$ is the set of all directed subsets of $X$).
    \item $\mathcal R$ (for $X\in$ \emph{ob}($\mathbf{Top}_0$), $\mathcal R(X)$ is the set of all irreducible subsets of $X$).
    \item $\mathsf{WD}$ (for $X\in$ \emph{ob}($\mathbf{Top}_0$), $\mathsf{WD}(X)$ is the set of all well-filtered determined subsets of $X$).
    \item $\mathsf{RD}$ (for $X\in$ \emph{ob}($\mathbf{Top}_0$), $\mathsf{RD}(X)$ is the set of all Rudin subsets of $X$).
\end{enumerate}

By Propositions \ref{2mt6} and \ref{2mt7}, $\mathsf{WD}$ and $\mathsf{RD}$ are R-subset systems; and by Proposition \ref{2mt6}, we have $\mathcal S\leq \mathcal D\leq \mathsf{RD}\leq \mathsf{WD}\leq \mathcal R$.

\begin{definition}\label{2dy9} (\cite{Xu2}) Let ${\rm H} : \mathbf{Top}_0 \longrightarrow \mathbf{Set}$ be an R-subset system and $X$ be a $T_{0}$ space. $X$ is called \emph{H}-\emph{sober} if for any $A\in {\rm H}(X)$, there is a (unique) point $x\in X$ such that $\overline{A}=\overline{\{x\}}$ or, equivalently, if ${\rm H}_c(X)=\mathcal S_c(X)$.
\end{definition}

\begin{definition}\label{2dy11} (\cite{Xu2})  Let ${\rm H} : \mathbf{Top}_0 \longrightarrow \mathbf{Set}$ be an R-subset system and $X$ a $T_0$ space.
A subset $A$ of $X$ is called \emph{H}-\emph{sober determined}, if for any continuous mapping $ f:X\longrightarrow Y$
to an H-sober space $Y$, there exists a unique $y_A\in Y$ such that $\overline{f(A)}=\overline{\{y_A\}}$. Denote by ${\rm H}^d(X)$ the set of all H-sober determined subsets of $X$. The set of all closed H-sober determined subsets of $X$ is denoted by ${\rm H}^d_c(X)$.
\end{definition}

\section{Function spaces related to H-sober spaces}

For $T_0$ spaces $X$ and $Y$, there are four important topologies on the set $\mathbb{C}(X, Y)$ of all continuous functions from $X$ to $Y$, namely, the pointwise convergence topology, compact-open topology, Isbell topology and Scott topology. In this section, we will discuss all of these topologies on $\mathbb{C}(X, Y)$.

We begin by recalling the definitions of pointwise convergence topology, Isbell topology and compact-open topology. For further details, we refer the reader to \cite{RE, GHKLMS, GL}.

In the following, $\mathbb{C}(X, Y)$ always means the set of all continuous functions from $X$ to $Y$.

\begin{definition}\label{3dy1} Let $X$ and $Y$ be topological spaces.
\begin{enumerate}[\rm (1)]
\item For a point $x\in X$ and an open set $U\in \mathcal{O}(Y)$, let $S(x, U)=\{f\in \mathbb{C}(X, Y)$ : $f(x)\in U\}$. The set $\{S(x, U): x\in X, U\in \mathcal{O}(Y)\}$ is a subbasis for the \emph{pointwise convergence topology} (i.e.,
the relative product topology) on $\mathbb{C}(X, Y)$. Let $[X\rightarrow Y]_{P}$ denote the function space $\mathbb{C}(X, Y)$ endowed with the topology of pointwise convergence.
\item The \emph{Isbell topology} on the set $\mathbb{C}(X, Y)$ is generated by the subsets of the form $N(\mathcal H\leftarrow V)=\{f\in \mathbb{C}(X,Y):f^{-1}(V)\in \mathcal H\},$ where $\mathcal H$ is a Scott open subset of the complete lattice $\mathcal{O}(X)$ and $V$ is open in $Y$. Let $[X\rightarrow Y]_{I}$ denote the function space $\mathbb{C}(X, Y)$ endowed with the Isbell topology.
\item The \emph{compact-open topology} on the set $\mathbb{C}(X, Y)$ is generated by the subsets of the form $N(K\rightarrow V)=\{f\in \mathbb{C}(X,Y): f(K)\subseteq V\},$ where $K$ is compact in $X$ and $V$ is open in $Y$. Let $[X\rightarrow Y]_{K}$ denote the function space $\mathbb{C}(X, Y)$ endowed with the compact-open topology.
\end{enumerate}
\end{definition}

For a topological space $X$ and a $T_0$ space $Y$, $Y$ (with the specialization order) is a poset, whence $\mathbb{C}(X, Y)$ is a poset with the pointwise order. Denote by $[X\rightarrow Y]_{\Sigma}$ the Scott space $\Sigma \mathbb{C}(X, Y)$. It is easy to see that the specialization orders on $[X\rightarrow Y]_{P}$, $[X\rightarrow Y]_{K}$, $[X\rightarrow Y]_{I}$ and $[X\rightarrow Y]_{\Sigma}$ are all the usually pointwise order on $\mathbb{C}(X, Y)$, i.e., for any $f, g\in \mathbb{C}(X, Y)$, $$f\leq_{[X\rightarrow Y]_{P}}g ~\mbox{iff}~ f\leq_{[X\rightarrow Y]_{K}}g ~\mbox{iff}~ f\leq_{[X\rightarrow Y]_{I}}g ~\mbox{iff}~ f\leq_{[X\rightarrow Y]_{\Sigma}}g ~\mbox{iff}~ f(x)\leq_Y g(x) ~\mbox{for all}~ x\in X.$$

\begin{lemma}\emph{(\cite[Lemma II-4.2]{GHKLMS})}\label{3yl1} For $T_0$ spaces $X$ and $Y$, the Isbell topology is finer than the compact-open topology which in turn is finer than the topology of pointwise convergence. If $X$ is sober and $\mathcal{O}(X)$ is a continuous lattice, then the Isbell topology and the compact-open topology agree.
\end{lemma}

\begin{lemma}\label{3yl2} Let $X, Y$ be $T_0$ spaces and $x\in X$. Consider the function $$\varphi^{P}_x:[X\rightarrow Y]_{P}\longrightarrow Y,$$ $$\!~~~~~~~~~~~~~f\!~~~\mapsto \!~f(x).$$ Then $\varphi^{P}_x$ is continuous.
\end{lemma}
\begin{proof}
Suppose $f\in [X\rightarrow Y]_{P}$ and $\varphi^{P}_x(f)\in V$ (i.e., $f(x)\in V$), where $V$ is open in $Y$. Then $f\in S(x, V)$ and $S(x, V)$ is open in $[X\rightarrow Y]_{P}$. Clearly, $\varphi^{P}_x(S(x, V))\subseteq V$. So $\varphi^{P}_x:[X\rightarrow Y]_{P}\longrightarrow Y$ is continuous.
\end{proof}

\begin{corollary}\label{3tl3} Let $X, Y$ be $T_0$ spaces and $x\in X$. Then
\begin{enumerate}[\rm (1)]
    \item the function $\varphi^{K}_x:[X\rightarrow Y]_{K}\longrightarrow Y$, $f\mapsto f(x)$, is continuous.
    \item the function $\varphi^{I}_x:[X\rightarrow Y]_{I}\longrightarrow Y$, $f\mapsto f(x)$, is continuous.
\end{enumerate}
\end{corollary}
\begin{proof}
By Lemmas \ref{3yl1} and \ref{3yl2}.
\end{proof}

In the following, for topological spaces $X, Y$ and $y\in Y$, $c_y$ denotes the constant function from $X$ to $Y$ with value $y$, i.e., $c_{y}(x)=y$ for all $x\in X$.

\begin{lemma}\emph{(\cite[Lemma 3.2]{LLH})}\label{3yl4}
Let $X$ and $Y$ be $T_0$ spaces. Consider the function $$\!~~~~~\psi^{I}:Y\longrightarrow[X\rightarrow Y]_{I},$$ $$y\!~\mapsto \!~c_{y}.$$ Then $\psi^{I}$ is continuous.
\end{lemma}

By Lemmas \ref{3yl1} and \ref{3yl4}, we can directly obtain the following corollary.

\begin{corollary}\label{3tl5} Let $X$ and $Y$ be $T_0$ spaces. Then
\begin{enumerate}[\rm (1)]
    \item the function $\psi^{K}:Y\longrightarrow[X\rightarrow Y]_{K}$, $y\mapsto c_{y}$, is continuous.
    \item the function $\psi^{P}:Y\longrightarrow[X\rightarrow Y]_{P}$, $y\mapsto c_{y}$, is continuous.
\end{enumerate}
\end{corollary}

\begin{remark}\label{plus1} The similar result of Lemma \ref{3yl4} and Corollary \ref{3tl5} for the Scott topology does not hold in general. Indeed, let $X=1=\{0\}$ be the topological space with only one point (clearly, $\mathcal O (X)=\mathcal O(1)=\{\emptyset, 1\}$ and $Y$ the space $(P, \upsilon (P))$, where $P$ is a poset for which the Scott topology is strictly finer than the upper topology on $P$ (for example, $P$ is a countably infinite set with the discrete order). Then $\mathbb{C}(X, Y)=\{c_y : y\in Y\}$ and $c_y\mapsto y : [X\rightarrow Y]_{\Sigma}\rightarrow \Sigma Y=\Sigma P$ is a homeomorphism. Since $\upsilon (P)\subsetneqq \sigma (P)$, the function $\psi^{\Sigma}: Y\longrightarrow[X\rightarrow Y]_{\Sigma}$, $y\mapsto c_{y}$, is not continuous.

\end{remark}

\begin{lemma}\emph{(\cite[Proposition 4.28]{Xu2})}\label{3mt6}
A retract of an H-sober space is H-sober.
\end{lemma}

\begin{corollary}\label{3mt6-}
A retract of a sober space (resp., $d$-space, well-filtered space) is a sober space (resp., $d$-space, well-filtered space).
\end{corollary}

\begin{proposition} \label{3mt8}
Let $X, Y$ be $T_{0}$ spaces.
\begin{enumerate}[\rm (1)]
    \item If the function space $\mathbb{C}(X, Y)$ equipped with the compact-open topology is H-sober, then $Y$ is H-sober.
    \item If the function space $\mathbb{C}(X, Y)$ equipped with the pointwise convergence topology is H-sober, then $Y$ is H-sober.
\end{enumerate}
\end{proposition}
\begin{proof} (1): Suppose that $[X\rightarrow Y]_{K}$ is H-sober. Select an $x\in X$. Then for any $y\in Y$, $(\varphi^{K}_x \circ \psi^{K})(y)=\varphi^{K}_x (c_y)=y$ and $\varphi^{K}_x$, $\psi^{K}$ are continuous by Corollary \ref{3tl3} (1) and Corollary \ref{3tl5} (1). Thus $Y$ is a retract of $[X\rightarrow Y]_{K}$. Hence $Y$ is H-sober by Lemma \ref{3mt6}.

(2): The proof is similar to that of (1).

\end{proof}

\begin{theorem}\emph{(\cite[Theorem 4.30]{Xu2})}\label{3dl8} Let $H:\mathbf{Top}_{0}\longrightarrow \mathbf{Set}$ be an R-subset system, $X$ a $T_{0}$ space and $Y$ an H-sober space. Then the function space $\mathbb{C}(X, Y)$ equipped with the topology of pointwise convergence is H-sober.
\end{theorem}

By Proposition \ref{3mt8} and Theorem \ref{3dl8}, we have the following result.

\begin{corollary}\label{3dl9} Let $H:\mathbf{Top}_{0}\longrightarrow \mathbf{Set}$ be an R-subset system and $X, Y$ two $T_{0}$ spaces. Then the following two conditions are equivalent:
\begin{enumerate}[\rm (1)]
\item $Y$ is H-sober.
\item The function space $\mathbb{C}(X, Y)$ equipped with the topology of pointwise convergence is H-sober.
\end{enumerate}
\end{corollary}

Applying Proposition \ref{3mt8} and Corollary \ref{3dl9} directly to the R-subset systems $\mathcal{D}$, $\mathcal{R}$ and $\mathsf{WD}$ (or $\mathsf{RD}$), we get the following two corollaries.

\begin{corollary} \label{3mt8+}
Let $X, Y$ be $T_{0}$ spaces. If the function space $\mathbb{C}(X, Y)$ equipped with the compact-open topology is a sober space (resp., $d$-spaces, well-filtered space), then $Y$ is a sober space (resp., $d$-spaces, well-filtered space).
\end{corollary}

\begin{corollary}\label{3tl9} For $T_0$ spaces  $X$ and $Y$, the following two conditions are equivalent:
\begin{enumerate}[\rm (1)]
\item $Y$ is a sober space (resp., $d$-spaces, well-filtered space).
\item The function space $\mathbb{C}(X, Y)$ equipped with the pointwise convergence topology is a sober space (resp., $d$-space, well-filtered space).
\end{enumerate}
\end{corollary}

The results for $d$-spaces and sober spaces in Corollary \ref{3tl9} were first shown in \cite[Theorem 3 and Theorem 6]{ES} by a different method.

\begin{proposition}\label{3dl16+} (\cite[LemmaII-4.3]{GHKLMS})
For a $T_0$ $X$ and a $d$-space $Y$, the function space $\mathbb{C}(X, Y)$ equipped with the Isbell topology is a $d$-space and hence the Scott topology on $\mathbb{C}(X, Y)$ is finer than the Isbell topology.
\end{proposition}

From Lemma \ref{3yl1} and Proposition \ref{3dl16+} we deduce the following result.

\begin{corollary}\label{3tl15} For a $T_0$ space $X$ and a $d$-space $Y$, the function space $\mathbb{C}(X, Y)$ equipped with the compact-open topology is a $d$-space.
\end{corollary}

By Corollary \ref{3mt8+} and Corollary \ref{3tl15}, we get the following.

\begin{corollary}\label{3tl9+} For $T_0$ spaces $X$ and $Y$, the following two conditions are equivalent:
\begin{enumerate}[\rm (1)]
\item $Y$ is a $d$-space.
\item The function space $\mathbb{C}(X, Y)$ equipped with the compact-open topology is a $d$-space.
\end{enumerate}
\end{corollary}

\begin{remark} \label{3dl15+} We can give a direct proof of Corollary \ref{3tl15}.

\begin{proof}  Suppose that $X$ is a $T_0$ space and $Y$ is a $d$-space. We show that the function space $\mathbb{C}(X, Y)$ equipped with the compact-open topology is a $d$-space.

Let $\mathcal {F}\in \mathcal{D}([X\rightarrow Y]_{K})$. Since the specialization order on $[X\rightarrow Y]_{K}$ is the usually pointwise order on $\mathbb{C}(X, Y)$, we have that for each $x\in X$, $\{f(x): f\in \mathcal {F}\}\in \mathcal{D}(Y)$. As $Y$ is a $d$-space, there exists a unique point $a_{x}\in Y$ such that $\overline{\{f(x):f\in \mathcal {F}\}}=\overline{\{a_{x}\}}$. Then we can define a function $g:X\longrightarrow Y$ by $g(x)=a_{x}$ for each $x\in X$. It is straightforward to verify that $g$ is continuous (see Claim 1 in the proof of Theorem \ref{3dl11} below). Whence $g$ is an upper bound of $\mathcal{F}$ in $[X\rightarrow Y]_{K}$, and consequently, $\cl_{[X\rightarrow Y]_{I}}\mathcal {F}\subseteq \cl_{[X\rightarrow Y]_{I}}\{g\}$.

Now we show that $g\in \cl_{[X\rightarrow Y]_{K}}\mathcal {F}$. Let $N(K\rightarrow V)$ be a subbasis open set in $[X\rightarrow Y]_{K}$ with $g\in N(K\rightarrow V)$, where $K$ is a compact subset of $X$ and $V$ is an open subset of $Y$. Then for any $x\in K$, $g(x)\in V$. As $\overline{\{f(x):f\in \mathcal {F}\}}=\overline{\{g(x)\}}$, there exists $f\in \mathcal {F}$ with $f(x)\in V$, whence $x\in f^{-1}(V)$. Thus $K\subseteq \bigcup_{f\in \mathcal {F}}f^{-1}(V)$. By the compactness of $K$, there exists $\{f_{1}, f_{2},\cdots, f_{n}\}\subseteq F$ such that $K\subseteq \bigcup_{i=1}^{n}f_{i}^{-1}(V)$. Since $\mathcal {F}$ is directed in $[X\rightarrow Y]_{K}$, there is $h\in \mathcal {F}$ such that $f_{i}\leq h$ for all $i\in \{1, 2, \cdots, n\}$. Therefore, $K\subseteq \bigcup_{i=1}^{n}f_{i}^{-1}(V)\subseteq h^{-1}(V)$ (note that open sets in $Y$ are upper sets with the specialization order of $Y$), i.e., $h\in N(K\rightarrow V)$. So  $\mathcal {F}\cap N(K\rightarrow V)\neq \emptyset$. Since $\mathcal {F}\in \mathcal{D}([X\rightarrow Y]_{K})\subseteq \ir ([X\rightarrow Y]_{K})$ and all basic open sets of $g$ in $[X\rightarrow Y]_{K}$ must meet $\mathcal {F}$, we get that $g\in \cl_{[X\rightarrow Y]_{K}}\mathcal {F}$, and hence $\cl_{[X\rightarrow Y]_{K}}\mathcal {F}=\cl_{[X\rightarrow Y]_{K}}\{g\}$. Therefore, the function space $\mathbb{C}(X, Y)$ equipped with the compact-open topology is a $d$-space.

\end{proof}
\end{remark}

But we do not know whether the converse of Proposition \ref{3mt8} (1) holds (especially, whether Corollary \ref{3tl15} holds for well-filtered spaces and sober spaces). Therefore, we pose the following questions.

\begin{question} \label{3wt1}
For a $T_{0}$ space $X$ and an H-sober space $Y$, is the function space $\mathbb{C}(X, Y)$ equipped with the compact-open topology H-sober?
\end{question}

\begin{question} \label{3wt2}
For a $T_{0}$ space $X$ and a sober (resp., well-filtered) space $Y$, is the function space $\mathbb{C}(X, Y)$ equipped with the compact-open topology sober (resp., well-filtered)?
\end{question}

\begin{proposition} \label{3dl7}
For $T_{0}$ spaces $X$ and $Y$, if the function space $\mathbb{C}(X, Y)$ equipped with the Isbell topology is H-sober, then $Y$ is H-sober.
\end{proposition}
\begin{proof}
Suppose that $[X\rightarrow Y]_{\Sigma}$ is H-sober. Select an $x\in X$. Then for any $y\in Y$, $(\varphi^{I}_x \circ \psi^{I})(y)=\varphi^{I}_x (c_y)=y$ and $\varphi^{I}_x$, $\psi^{I}$ are continuous by Corollary \ref{3tl3} (2) and Lemma \ref{3yl4}. Thus $Y$ is a retract of $[X\rightarrow Y]_{I}$. Hence $Y$ is H-sober by Lemma \ref{3mt6}.
\end{proof}

\begin{lemma}\label{3yl9}
Let $X$, $Y$ be topological spaces and $V$ is open in $Y$. Consider the function $$\zeta_V:[X\rightarrow Y]_{I}\longrightarrow \Sigma \mathcal{O}(X),$$ $$\!~~~~~~~~~~~~~f\!~~~\mapsto \!~f^{-1}(V).$$ Then $\zeta_V$ is continuous.
\end{lemma}
\begin{proof}
Suppose $V\in \mathcal{O}(Y)$, $f\in [X\rightarrow Y]_{I}$ and $\zeta_V(f)=f^{-1}(V)\in \mathcal H$, where $\mathcal H$ is Scott open in $\mathcal{O}(X)$. Then $f\in N(\mathcal H\leftarrow V)$, $N(\mathcal H\leftarrow V)$ is open in $[X\rightarrow Y]_{I}$ and for any $g\in N(\mathcal H\leftarrow V)$, $\zeta_V(g)=g^{-1}(V)\in \mathcal H$, whence $\zeta_V (N(\mathcal H\leftarrow V))\subseteq \mathcal H$. It follows that $\zeta_V:[X\rightarrow Y]_{I}\longrightarrow \Sigma \mathcal{O}(X)$ is continuous.
\end{proof}

\begin{lemma} (\cite{Xu2})\label{3yl10}
Let $H:\mathbf{Top}_{0}\longrightarrow \mathbf{Set}$ be an R-subset system. Then for any $T_{0}$ space $X$, $P_{H}(H_{c}^{d}(X))$ is an H-sober space. In fact, for any $\{A_{i}:i\in I\}\in H(P_{H}(H_{c}^{d}(X)))$, $A=\overline{\bigcup_{i\in I}A_{i}}\in H_{c}^{d}(X)$ and $\overline{\{A_{i}:i\in I\}}=\overline{\{A\}}$ in $P_{H}(H_{c}^{d}(X))$. Moreover, $X^{h}=P_{H}(H_{c}^{d}(X))$ with the canonical topological embedding $\eta_{X}^{h}: X\longrightarrow X^{h}$ is the H-sobrification of $X$, where $\eta_{X}^{h}(x)=\overline{\{x\}}$ for all $x\in X$.
\end{lemma}

One of the main results of this paper is the following.

\begin{theorem} \label{3dl11}
Let $H:\mathbf{Top}_{0}\longrightarrow \mathbf{Set}$ be an R-subset system. Then for a $T_{0}$ space $X$ and an H-sober space $Y$, the function space $\mathbb{C}(X, Y)$ equipped with the Isbell topology is an H-sober space.
\end{theorem}

\begin{proof}
Let $\mathcal {F}\in H([X\rightarrow Y]_{I})$. Since the specialization order on $[X\rightarrow Y]_{I}$ is the usually pointwise order on $\mathbb{C}(X, Y)$, $[X\rightarrow Y]_{I}$ is $T_{0}$. For $x\in X$, by Corollary \ref{3tl3} (2), the function $\varphi^{I}_x:[X\rightarrow Y]_{I}\longrightarrow Y$,  $f\mapsto f(x)$, is continuous. So $H(\varphi^{I}_x)(\mathcal {F})=\varphi^{I}_x(\mathcal {F})=\{f(x): f\in \mathcal {F}\}\in H(Y)$. As $Y$ is H-sober, there is a unique point $a_{x}\in Y$ such that $\overline{\{f(x):f\in \mathcal {\mathcal {F}}\}}=\overline{\{a_{x}\}}$. Now we can define a function $$g:X\longrightarrow Y \ \  by  \ \  g(x)=a_{x} \hbox{~for~each~}  x\in X.$$

\smallskip
{\bf Claim 1:} $g:X\longrightarrow Y$ is continuous.

\smallskip
Let $x\in X$ and $V\in \mathcal{O}(Y)$ with $g(x)=a_{x}\in V$. Then by $\overline{\{a_{x}\}}=\overline{\{f(x):f\in \mathcal {\mathcal {F}}\}}=$, we have $\{f(x):f\in \mathcal {F}\}\bigcap V\neq \emptyset$, whence there is $f\in \mathcal {F}$ such that $f(x)\in V$. As $f:X\longrightarrow Y$ is continuous, there is a $U\in \mathcal{O}(X)$ with $x\in U$ and $f(U)\subseteq V$. For any $z\in U$, $f(z)\in \overline{\{f(z): f\in \mathcal {F}\}}=\overline{\{a_{z}\}}=\overline{\{g(z)\}}$, and hence $g(z)\in V$ (note that $V$ is an upper set in $Y$ with the specialization order). So $g(U)\subseteq V$. Thus $g$ is continuous.

\smallskip
{\bf Claim 2:} For any subbasis open set $N(\mathcal H\leftarrow W)$ in $[X\rightarrow Y]_{I}$ with $g\in N(\mathcal H\leftarrow W)$, where $\mathcal H$ is a Scott open subset of $\mathcal{O}(X)$ and $W$ is an open subset of $Y$, we have $\mathcal {F}\bigcap N(\mathcal H\leftarrow W)\neq \emptyset$.

\smallskip
Since $g\in N(\mathcal H\leftarrow W)$, $W\in \mathcal{O}(Y))$, we have $g^{-1}(W)\in \mathcal H$. For any $x\in g^{-1}(W)$ (i.e.,  $g(x)\in W$), by $g(x)=a_{x}\in \overline{\{f(x): f\in \mathcal {F}\}}$, there exists $f_{0}\in \mathcal {F}$ such that $f_{0}(x)\in W$. Then $x\in f_{0}^{-1}(W)\subseteq \bigcup_{f\in \mathcal {F}}f^{-1}(W)$. So $g^{-1}(W)\subseteq \bigcup_{f\in \mathcal {F}}f^{-1}(W)$. As $\mathcal H$ is an upper set, we have $\bigcup_{f\in \mathcal {F}}f^{-1}(W)\in \mathcal H$.

\smallskip
Since $\Sigma \mathcal{O}(X)$ is $T_{0}$, one can directly deduce that $P_{H}(H_{c}^{d}(\Sigma \mathcal{O}(X)))$ is $T_{0}$. By Lemmas \ref{3yl9} and \ref{3yl10}, for each continuous function $f: X\longrightarrow Y$, $$\zeta_W:[X\rightarrow Y]_{I}\longrightarrow \Sigma \mathcal{O}(X),$$ $$\!~~~~~~~~~~~~~f\!~~~\mapsto \!~f^{-1}(W),$$ is continuous, and for any $U\in \mathcal{O}(X)$, $$\eta_{\Sigma \mathcal{O}(X)}^{h}: \Sigma \mathcal{O}(X)\longrightarrow P_{H}(H_{c}^{d}(\Sigma \mathcal{O}(X))),$$ $$\!~~~~~~~~~~~U\!~~\mapsto\!~~\cl_{\Sigma \mathcal{O}(X)}\{U\},$$ is continuous. So for each continuous function $f: X\longrightarrow Y$, $$\eta_{\Sigma \mathcal{O}(X)}^{h}\circ\zeta_W:[X\rightarrow Y]_{I}\longrightarrow P_{H}(H_{c}^{d}(\Sigma \mathcal{O}(X))),$$ $$\!~~~~~~~~~~~~~f\!~~~\mapsto \!~\cl_{\Sigma\mathcal{O}(X)}\{f^{-1}(W)\},$$ is continuous. Thus $\{\cl_{\Sigma\mathcal{O}(X)}\{f^{-1}(W)\}:f\in \mathcal {F}\}\in H(P_{H}(H_{c}^{d}(\Sigma \mathcal{O}(X))))$. By Lemma \ref{3yl10} again, $$\cl_{\Sigma\mathcal{O}(X)}\bigcup_{f\in \mathcal {F}}\cl_{\Sigma\mathcal{O}(X)}\{f^{-1}(W)\}=\cl_{\Sigma\mathcal{O}(X)}\{\bigcup_{f\in \mathcal {F}}f^{-1}(W)\}\in H_{c}^{d}(\Sigma\mathcal{O}(X))$$ and in $P_{H}(H_{c}^{d}(\Sigma \mathcal{O}(X)))$, $$\overline{\{\cl_{\Sigma\mathcal{O}(X)}\{f^{-1}(W)\}:f\in \mathcal {F}\}}=\overline{\{\cl_{\Sigma\mathcal{O}(X)}\{\bigcup_{f\in \mathcal {F}}f^{-1}(W)\}\}}.$$ As $\bigcup_{f\in \mathcal {F}}f^{-1}(W)\in \mathcal H$,  $\cl_{\Sigma\mathcal{O}(X)}\{\bigcup_{f\in \mathcal {F}}f^{-1}(W)\}\in \Diamond \mathcal H$. And since $\Diamond \mathcal H$ is open in $P_{H}(H_{c}^{d}(\Sigma \mathcal{O}(X)))$, we have that $\{\cl_{\Sigma\mathcal{O}(X)}\{f^{-1}(W)\}:f\in \mathcal {F}\}\bigcap \Diamond \mathcal H\neq \emptyset$. It follows that there exists $h\in \mathcal {F}$ such that $\cl_{\Sigma\mathcal{O}(X)}\{h^{-1}(W)\}\in \Diamond \mathcal H$. Then $\cl_{\Sigma\mathcal{O}(X)}\{h^{-1}(W)\} \bigcap \mathcal H\neq \emptyset$, implying that $h^{-1}(W)\in \mathcal H$. Therefore, $\mathcal {F}\bigcap N(\mathcal H\leftarrow W)\neq\emptyset$.

\smallskip
{\bf Claim 3:} $\cl_{[X\rightarrow Y]_{I}}\mathcal {F}=\cl_{[X\rightarrow Y]_{I}}\{g\}.$

\smallskip
As the specialization order on $[X\rightarrow Y]_{I}$ is the usually pointwise order on $\mathbb{C}(X, Y)$, $\cl_{[X\rightarrow Y]_{I}}\mathcal F\subseteq \cl_{[X\rightarrow Y]_{I}}\{g\}$ by the definition of $g$. On the other hand, by Claim 2 and $\mathcal {F}\in H([X\rightarrow Y]_{I})\subseteq \ir ([X\rightarrow Y]_{I})$, all basic open sets of $g$ in $[X\rightarrow Y]_{I}$ must meet $\mathcal {F}$. So $g\in \cl_{[X\rightarrow Y]_{I}}\mathcal {F}$ or, equivalently, $ \cl_{[X\rightarrow Y]_{I}}\{g\}\subseteq  \cl_{[X\rightarrow Y]_{I}}\mathcal {F}$. Thus $\cl_{[X\rightarrow Y]_{I}}\mathcal {F}=\cl_{[X\rightarrow Y]_{I}}\{g\}$.

\smallskip
By Claim 3, the function space $\mathbb{C}(X, Y)$ equipped with the Isbell topology is an H-sober space.
\end{proof}

By Proposition \ref{3dl7} and Theorem \ref{3dl11}, we get the following result.

\begin{theorem}\label{3dl12} For $T_0$ spaces $X$ and $Y$, the following two conditions are equivalent:
 \begin{enumerate}[\rm (1)]
 \item $Y$ is an H-sober space.
 \item The function space $\mathbb{C}(X, Y)$ equipped with the Isbell topology is an H-sober space.
\end{enumerate}
\end{theorem}

As the $d$-spaces, sober spaces and well-filtered spaces are three special types of H-sober spaces, by Theorem \ref{3dl12}, we get the following corollary.

\begin{corollary}\label{3dl13} For $T_0$ spaces $X$ and $Y$, the following two conditions are equivalent:
 \begin{enumerate}[\rm (1)]
 \item $Y$ is a sober space (resp., $d$-space, well-filtered space).
 \item The function space $\mathbb{C}(X, Y)$ equipped with the Isbell topology is a sober space (resp., $d$-space, well-filtered space).
\end{enumerate}
\end{corollary}

In \cite[Corollary 3.12]{LLH}, it was proved that for a core compact space $X$, if $Y$ is well-filtered, then the function space $\mathbb{C}(X, Y)$ equipped with the Isbell topology is a well-filtered space. For the well-filteredness, Corollary \ref{3dl13} improves this conclusion by removing the unnecessary condition that $X$ is core compact.

Finally, we investigate the Scott topology on $\mathbb{C}(X, Y)$.

\begin{lemma}\label{3dl16} (\cite[Lemma II-3.14 and Proposition II-3.15]{GHKLMS})
Let $X$ be a $T_0$ space and $Y$ a $d$-space. Then
\begin{enumerate}[\rm (1)]
\item if $(f_d)_{d\in D}$ is a net of continuous functions $f_d : X\rightarrow Y$ such that $(f_d(x))_{d\in D}$ is a
directed net of $Y$ (with the specialization order) for each $x\in X$, then the pointwise sup $f : X\rightarrow Y$ of the net
$f_d$ is continuous.
\item the subset $\mathbb{C}(X, Y)$ of
$(\Omega Y)^X$ is closed under directed sups and hence $\mathbb{C}(X, Y)$ is a dcpo.
\end{enumerate}
\end{lemma}

\begin{corollary}\label{3dl16++}
For a $T_0$ $X$ and a $d$-space $Y$, the function space $\mathbb{C}(X, Y)$ equipped with the Scott topology is a $d$-space.
\end{corollary}

The following example shows that for sober spaces, Corollary \ref{3dl16++} does not hold in general.

\begin{example}\label{plus4} Let $X=1$ be the topological space with only one point and $L$ the complete lattice constructed by Isbell in \cite{Isbell}. It is well-known that $\Sigma L$ is non-sober. Let $Y=(L, \upsilon (L))$. Then by \cite[Corollary 4.10]{ZhaoHo} or \cite[Proposition 2.9]{XSXZ}, $Y$ is sober. Clearly,  $\mathbb{C}(X, Y)=\{c_y : y\in Y\}$ and $c_y\mapsto y : [X\rightarrow Y]_{\Sigma}\rightarrow \Sigma \Omega Y=\Sigma L$ is a homeomorphism. So $Y$ is sober but the function space $\mathbb{C}(X, Y)$ equipped with the Scott topology is non-sober.

\end{example}

It is still not known whether Corollary \ref{3dl16++} holds for well-filtered spaces. That is, we have the following question.

\begin{question} \label{3wt3}
For a $T_{0}$ space $X$ and a well-filtered space $Y$, is the function space $\mathbb{C}(X, Y)$ equipped with the Scott topology well-filtered?
\end{question}

Let $\mathbb{J}=\mathbb{N}\times (\mathbb{N}\cup \{\omega\})$ with ordering defined by $(j, k)\leq (m, n)$ if{}f $j = m$ and $k \leq n$, or $n =\omega$ and $k\leq m$. $\mathbb{J}$ is a well-known dcpo constructed by Johnstone in \cite{johnstone-81} which is not sober in its Scott topology.

\begin{figure}[ht]
	\centering
	\includegraphics[height=4.5cm,width=4.5cm]{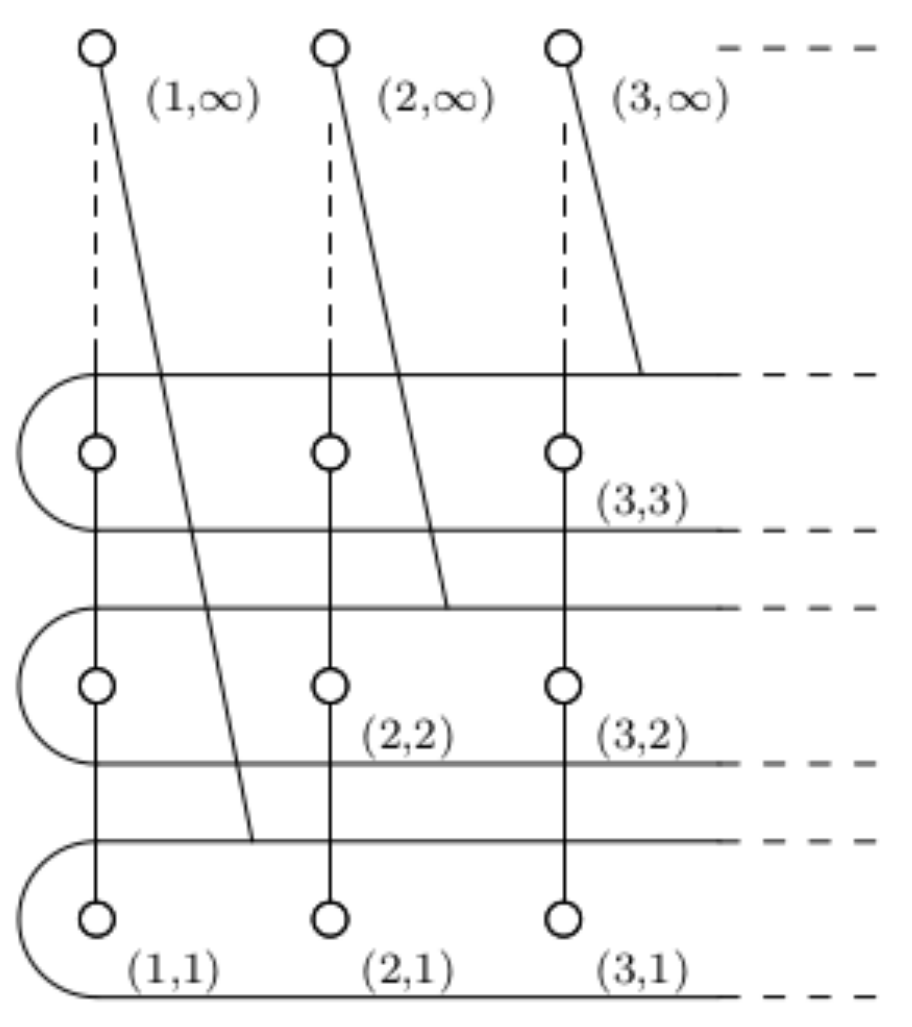}
	\caption{Johnstone's dcpo $\mathbb{J}$}
\end{figure}

\begin{proposition}\label{plus5} There is no well-filtered topology on $\mathbb{J}$ which has the given order as its specialization order, namely, for any topology $\upsilon (\mathbb{J})\subseteq \tau \subseteq \sigma (\mathbb{J})$, $(\mathbb{J}, \tau)$ is not well-filtered.
\end{proposition}
\begin{proof}

The set $\mathbb{J}_{max}=\{(n, \infty) : n\in\mn \}$ is the set of all maximal elements of $\mathbb{J}$ and $\mathsf{K}(\Sigma~\!\!\mathbb{J})=(2^{\mathbb{J}_{max}} \setminus \{\emptyset\})\bigcup \mathbf{Fin}~\!\mathbb{J}$ (see \cite[Example 3.1]{LL}). Since $\upsilon (\mathbb{J})\subseteq \tau \subseteq \sigma (\mathbb{J})$, we have that $2^{\mathbb{J}_{max}}\setminus \{\emptyset\}\subseteq \mathsf{K}((\mathbb{J}, \tau))$. Let $\mathcal K=\{\mathbb{J}_{max}\setminus F : F\in (\mathbb{J}_{max})^{(<\omega)}\}$. Then $\mathcal K\subseteq \mathsf{K}((\mathbb{J}, \tau))$ is a filtered family and $\bigcap\mathcal{K}=\bigcap_{F\in (\mathbb{J}_{max})^{(<\omega)}} (\mathbb{J}_{max}\setminus F)=\mathbb{J}_{max}\setminus \bigcup (\mathbb{J}_{max})^{(<\omega)})=\emptyset$, but $\mathbb{J}_{max}\setminus F\neq\emptyset$ for all $F\in (\mathbb{J}_{max})^{(<\omega)}$. Therefore, $(\mathbb{J}, \tau)$ is not well-filtered.
\end{proof}

\begin{corollary}\label{plus6} (\cite[Exercise II-3.16 (V)]{GHKLMS}) There is no sober topology on $\mathbb{J}$  which has the given order as its specialization order.
\end{corollary}

\begin{question} \label{3wt0}
Is there a well-filtered space $Y$ such that $\Sigma Y$ (i.e., $\Sigma \Omega Y$) is not well-filtered? Or equivalently, is there a dcpo $P$ and a topology $\upsilon (P)\subseteq \tau\subseteq \sigma (P)$ such that $(P, \tau)$ is well-filtered but $(P, \sigma (P))$ is not well-filtered?
In particular, is there a dcpo $P$ such that $(P, \upsilon (P))$ is well-filtered but $(P, \sigma (P))$ is not well-filtered?
\end{question}

If the answer of Question \ref{3wt0} is ``Yes", then the answer of Question \ref{3wt3} is ``No"! Indeed, suppose that $Y$ is a well-filtered space for which the Scott space $\Sigma Y$ is not well-filtered and $X$ is the topological space with only one point. Then the function space $\mathbb{C}(X, Y)$ equipped with the Scott topology is not well-filtered since $[X\rightarrow Y]_{\Sigma}$ and $\Sigma \Omega Y$ are homeomorphic (see Example \ref{plus4}).

Conversely, for $\mathcal R\geq {\rm H}\geq \mathcal D$ (in particular, ${\rm H}=\mathcal D, \mathcal R, \mathsf{WD}$ or $\mathsf {RD}$), the following example shows that the H-sobriety of function space $\mathbb{C}(X, Y)$ equipped with the Scott topology does not implies the H-sobriety of $Y$ in general.

\begin{example}\label{xu1} Let $X$ be the topological space with only one point and $[0, 1]$ the unit closed interval with the usual order of reals. Then $\sigma ([0, 1])\neq \gamma ([0, 1])$ (note that $\{1\}\in \gamma ([0, 1])$ but $\{1\}\not\in \sigma ([0, 1])$). Clearly, the Alexandroff space $\Gamma [0, 1]$ is not a $d$-space (since $\gamma ([0, 1])\nsubseteq\sigma ([0, 1])$) and the Scott space $\Sigma [0, 1]$ is sober, whence $\Sigma [0, 1]$ is H-sober and $\Gamma [0, 1]$ is not H-sober by $\mathcal R\geq {\rm H}\geq \mathcal D$. As $[X\rightarrow \Gamma [0, 1]]_{\Sigma}$ is homeomorphic to $\Sigma \Omega (\Gamma [0, 1])=\Sigma [0, 1]$ (see Example \ref{plus4}), the function space $\mathbb{C}(X, \Gamma [0, 1])$ equipped with the Scott topology is H-sober but $\Gamma [0, 1]$ is not H-sober. \end{example}

\end{document}